\documentclass[12pt]{amsart}
\usepackage[foot]{amsaddr}
\usepackage[margin=1in]{geometry} 
\usepackage{amsmath,amsthm,amssymb,amsfonts,physics,etoolbox,tikz,tensor}
\usepackage[backend=biber,style=alphabetic,url=false,isbn=false,doi=false,maxbibnames=99]{biblatex}
\AtEveryBibitem{
    \clearfield{urlyear}
    \clearfield{urlmonth}
    \clearfield{day}
}

\newtheorem{theorem}{Theorem}[section]
\newtheorem{proposition}[theorem]{Proposition}
\newtheorem{corollary}[theorem]{Corollary}
\newtheorem{lemma}[theorem]{Lemma}

\newtheorem{definition}[theorem]{Definition}

\newtheorem{conjecture}[theorem]{Conjecture}
\newtheorem{remark}[theorem]{Remark}

\newcommand{\g}{\mathfrak{g}}
\newcommand{\V}{{V^1(\mathfrak{psl}_{n|n})}}
\newcommand{\Omin}{\mathbb{O}_{\mathrm{min}}}
\newcommand{\sli}{\mathfrak{sl}}
\newcommand{\C}{\mathbb{C}}
\newcommand{\Z}{\mathbb{Z}}
\newcommand{\Sh}{\mathbb{S}_{\mathrm{min}}}
\newcommand{\levi}{\mathfrak{l}}
\newcommand{\para}{\mathfrak{p}}
\newcommand{\psl}{\mathfrak{psl}}
\theoremstyle{remark}
\newtheorem{example}{Example}

\title{$L_1(\mathfrak{psl}_{n|n})$ from BRST reductions, associated varieties and nilpotent orbits}
\author{Andrea E. V. Ferrari$^{\clubsuit,\diamondsuit}$}
\author{Aiden Suter$^{\spadesuit\heartsuit}$}
\address{$^\clubsuit$Deutsches
Elektronen-Synchrotron DESY, Notkestraße 85, 22607 Hamburg, Germany}
\address{$^\diamondsuit$School of Mathematics, The University of Edinburgh, Mayfield Road, Edinburgh, U.K}
\address{$^\spadesuit$Perimeter Institute for Theoretical Physics, 31 Caroline St N, Waterloo, ON N2L 2Y5, Canada}
\address{$^\heartsuit$Department of Pure Mathematics, University of Waterloo, 200 University Ave W, Waterloo, ON N2L 3G1, Canada}
\email{andrea.e.v.ferrari@gmail.com~and~aidensuter@outlook.com}

\begin{document}

\vspace*{-3.72em}
\begin{flushright}
DESY-24-132
\end{flushright}

\begin{abstract}

We verify a conjecture of~\cite{beemFreeFieldRealisation2023} stating that a certain family of physically motivated BRST reductions of beta-gamma systems and free fermions is isomorphic to $L_1(\psl_{n|n})$, and that its associated variety is isomorphic as a Poisson variety to the minimal nilpotent orbit closure $\overline{\Omin(\sli_n)}$. This shows in particular that $L_1(\psl_{n|n})$ is quasi-lisse. Combining this with other results in the literature (in particular~\cite{ballin3dMirrorSymmetry2023}), this paper provides a concrete and important example of how one can extract two symplectic dual varieties from a rather well-known vertex operator algebra.
\end{abstract}

\maketitle

\tableofcontents

\section{Introduction}

3-dimensional supersymmetric quantum field theories (QFT) with eight supercharges provide the natural context for a conjectural mathematical duality between conical symplectic singularities, called symplectic duality~\cite{bradenQuantizationsConicalSymplectic2022,bradenQuantizationsConicalSymplectic2022a}. The starting point is that a generic theory $\mathcal{T}$ has two distinct branches of its moduli spaces of vacua, the Higgs branch and Coulomb branch, which under some assumptions are conjecturally symplectic dual pairs. In the case of a gauge theory, which is determined by data $(G,R)$ constituted by a (compact, connected) gauge group $G$ and a quaternionic representation $R$, the Higgs branch $M_H$ can be defined as a hyperk\"ahler quotient
\begin{equation*}\label{eq:higgs}
    M_H = R /\!\!/\!\!/ G
\end{equation*}
or alternatively (picking a complex structure and forgetting about the hyperk\"ahler metric) as a holomorphic symplectic quotient
\begin{equation}\label{eq:higgs2}
 M_H = \mu_{\mathbb{C}}^{-1}(0) /\! \! /G_{\mathbb{C}}.
\end{equation}
Here $\mu_{\mathbb{C}}$ is the complex moment map for the action of $G$ on $R$ and $G_{\mathbb{C}}$ is a complexification of the group $G$. In the case $R=T^*V$ for some complex representation $V$, the Coulomb branch is given instead by the famous BFN construction~\cite{Braverman:2016wma}. One of the most elementary examples, and the one relevant for this paper, arises from so-called supersymmetric QED with $n\geq 2$ hypermultiplets, in short SQED$[n]$. In this case $G=U(1)$ and $R=T^* \mathbb{C}^n$ where $\mathbb{C}$ is the fundamental representation of $U(1)$. The Higgs branch of SQED$[n]$ can be easily computed to be $\overline{\Omin(\sli_n)}$, whereas its Coulomb branch is the $A_{n-1}$ singularity.\\

A lot of insights on symplectic duality can be gained by inspecting topological twists of such a theory $\mathcal{T}$. For any $\mathcal{T}$ there exists two such twists, the A-~and~B-twist, and therefore two cohomological TQFTs $\mathcal{T}^A$ and $\mathcal{T}^B$. In the bulk of the paper we will simply refer to these two twists as ``3d A-model" and ``3d B-model", respectively. The most elementary observation relevant for symplectic duality in this TQFT context is that the topological local operators of $\mathcal{T}^A$ are closely related to the ring of functions on the Coulomb branch, $\mathbb{C}[M_C]$, whereas topological local operators of $\mathcal{T}^B$ are closely related to the rings of functions on the Higgs branch, $\mathbb{C}[M_H]$.\footnote{It is worth noting that 3d Mirror Symmetry is aimed at establishing isomorphisms between the 3d A-model and B-model of pairs of theories $\mathcal{T}$ and $\mathcal{T} ^\vee$. In particular, the Higgs/Coulomb branch of a theory $\mathcal{T}$ is isomorphic to the Coulomb/Higgs branch of the other theory $\mathcal{T}^\vee$. We leave the investigation of our results in light of mirror symmetry to future work.}\\

A great deal of progress in understanding $\mathcal{T}^{A/B}$ has been made by studying boundary vertex operator algebras (VOAs). By considering the 3d A-~and~B-model of a gauge theory $\mathcal{T}^{A/B}$ on the geometry $M=\mathbb{R}^+\times \mathbb{C}$ (endowed with suitable boundary conditions), two VOAs supported at the boundary $\partial M = \{ 0 \} \times \mathbb{C}$ may be constructed \cite{gaiottoTwistedCompactifications3d2018,costelloHiggsCoulombBranches2018,costelloVertexOperatorAlgebras2019}. We simply denote these by $V^A$ and $V^B$ (and decorate them by labels that specify the theory, whenever necessary). Importantly, in the A-twist it is often necessary to introduce additional boundary free fermions to define $V^A$. When this is done $V^A$ can be identified with a relative BRST reduction modelled on the holomorphic symplectic quotient~\eqref{eq:higgs} used to defined the Higgs branch, but with the free fermions included (see Definition~\ref{def:VA} below for more details). One natural question to ask is then how the two VOAs are related to the Higgs and Coulomb branch of the theory, and whether one can extract symplectic dual pairs from a given VOA.\\

It was already conjectured in~\cite{costelloHiggsCoulombBranches2018}~and~\cite{costelloVertexOperatorAlgebras2019} that (provided the boundary condition supports enough local operators, in a technical sense) the derived endomorphism algebra the tensor unit in a suitable category of modules for $V^A/V^B$ is isomorphic to the ring of local operators of the respective TQFT, and therefore closely related to algebra of functions on the Coulomb/Higgs branch. It was further conjectured that certain modules for the boundary VOA would correspond to bulk line operators in $\mathcal{T}^{A/B}$. The structure of this category of modules and many aspects of this conjecture were proven for abelian gauge theories (that is, gauge theories where $G$ is a torus) in \cite{ballin3dMirrorSymmetry2023}. In particular, in the special case of SQED[$n$] we can immediately verify (see Theorem~8.32 and Proposition~8.39 therein)\\
\begin{theorem}[\mbox{\cite{ballin3dMirrorSymmetry2023}}]\label{tudproof}
For SQED[$n$], with $\mathbf{1}^A\in \mathcal{C}^A$ the tensor unit in a suitable category of modules $\mathcal{C}^A$ for $V^A$, we have
 \begin{align*}
     \mathrm{Ext}^\bullet \left(\mathbf{1}^A,\mathbf{1}^A\right) \cong \mathbb{C}[A_{n-1}].&
 \end{align*}
\end{theorem}

Another natural way to extract a variety from $V^A/V^B$ that is at least Poisson from a VOA is of course to consider its associated variety~\cite{arakawaRemarkC_2Cofiniteness2010}. In the more familiar context of VOAs arising form 4d $\mathcal{N}=2$ SCFTs, a fruitful conjecture of Beem and Rastelli~\cite{Beem:2017ooy} states that the associated variety is isomorphic to the Higgs branch. In work of Beem and one of the authors~\cite{beemFreeFieldRealisation2023}, based on observations of~\cite{costelloVertexOperatorAlgebras2019}, the tantalising form of the BRST reduction defining the VOAs as well as considerations about dimensional reductions of the set-up  of Beem-Rastelli\footnote{See~\cite{Dedushenko:2023cvd} for a more recent and detailed study of this reduction.}, the following conjecture was formulated
\begin{conjecture}[\mbox{\cite{beemFreeFieldRealisation2023}}]\label{mainconjecture}
For an abelian gauge theory in the A-twist with free fermions as additional boundary degrees of freedom and boundary VOA $V^A$, the associated variety $X_{V^A}$ is isomorphic to the Higgs branch of the theory.
\end{conjecture}
The conjecture was substantiated in several ways, \emph{in primis} by considering free field realisations modelled on the Higgs branch geometry. It has some striking consequences. Since as we mentioned above $V^A$ is defined in terms of a BRST reduction modelled on the quotient~\eqref{eq:higgs}, this would imply that in these examples the principle ``chiralisation commutes with reduction" holds true. Moreover, since the Higgs branches have finitely many symplectic leaves, it would provide a recipe to construct potentially new examples of quasi-lisse VOAs.\footnote{The statement that such a conjecture holds true for non-abelian theories as well was subsequently presented and motivated in~\cite{Coman:2023xcq}.} \\

More precise statements were conjectured in~\cite{beemFreeFieldRealisation2023} for SQED[$n$]. In particular, it was conjectured (based on observations of~\cite{costelloVertexOperatorAlgebras2019}) that
\begin{conjecture}[\mbox{\cite{beemFreeFieldRealisation2023}}]\label{mainconjecture2}
For SQED[$n$], we have
\begin{enumerate}
    \item  \begin{align*}
        &&V^A&\cong L_1(\psl_{n|n})&
    \end{align*}
 and so, by Conjecture~\ref{mainconjecture}
\item  \begin{align*}
    &&X_{L_1(\psl_{n|n})}&\cong \overline{\Omin(\sli_n)}.&
\end{align*}
\end{enumerate}
\end{conjecture}
The following special case was moreover proved:
\begin{theorem}[\mbox{\cite{beemFreeFieldRealisation2023}}]\label{bfproof}
For $n=2$ we have
 \begin{align*}
     &&X_{L_1(\psl_{n|n})}&\cong \overline{\Omin(\sli_2)}.&
 \end{align*}
\end{theorem}

Unfortunately, the technique used in ~\cite{beemFreeFieldRealisation2023} to prove Theorem \ref{bfproof} does not readily generalise for $n>2$. In this paper we manage to circumvent this issue and to prove
\begin{theorem}
Conjecture~\ref{mainconjecture2} holds for $n > 2$.
\end{theorem}
Together with Theorem~\ref{tudproof}, this shows how to extract the symplectic dual pair $\overline{\Omin(\sli_n)}$ and $A_{n-1}$ from $V^A \cong L_1 (\psl_{n|n})$. Furthermore, this implies the interesting corollary
\begin{corollary}
    For $n \geq 2$, $L_1(\psl_{n|n})$ is quasi-lisse.
\end{corollary}

The proof strategy is as follows. A general expression for the BRST cohomology defining $V^A$ for abelian gauge theories in terms of free field realisations was given in~\cite{ballin3dMirrorSymmetry2023,beemFreeFieldRealisation2023}, and in~\cite{ballin3dMirrorSymmetry2023} this was recast in terms of simple current extensions. Combining this with results expressing $L_1(\psl_{n|n})$ as a simple current extension of $L_1(\sli_n)\otimes L_{-1}(\sli_n)$ \cite{adamovicConformalEmbeddingsAffine2019,creutzigTensorCategoriesAffine2021} makes proving the first statement of the conjecture relatively straightforward. Proving the second statement is at first glance quite challenging, as the submodule structure of $V^1(\psl_{n|n})$, and hence the structure of $L_1(\psl_{n|n})$, is not well-understood. However, we manage to find a singular vector generating a submodule $U$ of $V^1(\psl_{n|n})$ that contains the maximal submodule of the subalgebra $V^1(\sli_n)\otimes V^{-1}(\sli_n)$ for $n>3$, and of the left factor for all $n>1$. This implies (as all odd elements are nilpotent) that the associated variety is an $SL_n$-invariant subvariety of $\sli_n^*$, and in fact for $n>3$ by results due to Arakawa and Moreau a subvariety of the unique minimal sheet containing the desired minimal nilpotent orbit~\cite{arakawaSheetsAssociatedVarieties2019}. The sheet is a deformation of this nilpotent orbit as a symplectic singularity, which also contains elements that are semisimple elements of $\sli_n^*$. We then show that for all $n \geq 2$ the submodule $U$ contains elements that at the level of the reduced $C_2$-algebra produce all $2\times 2$ minors of $\sli_n$, including those that for $n>3$ eliminate the undesired elements of the sheet. Since $L_1(\psl_{n|n})$ is not lisse and the minimal nilpotent orbit closure does not contain any non-trivial $SL_n$-invariant subvarieties, the result follows. We conjecture that $U$ may correspond to the maximal submodule of $V^1(\psl_{n|n})$, however leave this to future work.

\subsection{Acknowledgements} It is our pleasure to thank Dr\v{a}zen Adamovi\'c, Christopher Beem, Thomas Creutzig, Tudor Dimofte, Anne Moreau, Wenjun Niu, David Ridout, Volker Schomerum, and especially Ben Webster for useful comments and discussions. The work of AEVF was supported in part by the EPSRC Grant EP/W020939/1 ``3d N=4 TQFTs". AS was supported by a MITACS Globalink Fellowship and thanks the University of Edinburgh for its hospitality. AS was also supported in part by Perimeter Institute for Theoretical Physics. Research at Perimeter Institute is supported in part by the Government of Canada through the Department of Innovation, Science and Economic Development Canada and by the Province of Ontario through the Ministry of Colleges and Universities.

\section{Preliminaries}

\subsection{Notation \& conventions}\label{notation}
Throughout this paper all vector spaces are over $\C$ and we use the following convention
\begin{align*}
    &&\mathbb{N}&=\lbrace 0,1,2,\dots \rbrace, &\mathbb{Z}^\pm&=\lbrace \pm 1,\pm 2\, \pm 3,\dots\rbrace.&
\end{align*}
We denote by $\g$ a finite-dimensional Lie superalgebra with superspace decomposition
\begin{align*}
    &&\g&=\g_{\bar{0}}\oplus\g_{\bar{1}}&
\end{align*}
and a triangular decomposition
\begin{align*}
    &&\g&=\mathfrak{n}_+\oplus \mathfrak{h}\oplus\mathfrak{n}_-.&
\end{align*}
We will denote by $G$ its associated Lie supergroup with $\mathrm{Lie}(G)=\g$ and by $\widehat{\g}$ the its affine Lie algebra. As a vector space this decomposes as
\begin{align*}
    &&\widehat{\g}&=\g(\!(t)\!)\oplus\C K\oplus \C T &
\end{align*}
where $K$ is central and $T$ acts as a derivation with respect to $t$. We will denote the triangular decomposition of $\widehat{\g}$ induced from the above triangular decomposition of $\g$ by
\begin{align*}
    &&\widehat{\g}&=\widehat{\mathfrak{n}}_+\oplus\widehat{\mathfrak{h}}\oplus\widehat{\mathfrak{n}}_-&
\end{align*}
where
\begin{align*}
    &&\widehat{\mathfrak{n}}_\pm&=\mathfrak{n}_\pm\oplus (t^{\pm1} \g[\![t^{\pm1}]\!]),& \widehat{\mathfrak{h}}&=\mathfrak{h}\oplus \C K\oplus \C T .&
\end{align*}
Throughout this paper, we take $n \in \mathbb{N}$ and when not otherwise stated $n>1$. Let $E^{i,j}\in\mathrm{Mat}_{n\times n}(\C)$ denote the elementary matrix with a $1$ in the $(i,j)$\textsuperscript{th} position and zeroes elsewhere. We will make use of the following notation
\begin{align*}
    &&\varepsilon_i:\mathfrak{h}&\to \C&\\
    &&E^{i,j}&\mapsto\begin{cases}
        1, &i=j \\
        0, &i\neq j
    \end{cases}
\end{align*}
as well as
\begin{align*}
    &&\varepsilon_{i,j}&=\varepsilon_i-\varepsilon_j, &\alpha_i&=\varepsilon_{i,i+1}.&
\end{align*}
We fix the following notation for the roots, positive roots, simple roots and simple coroots of $\sli_n$ respectively:
\begin{align*}
    &&\Delta&=\lbrace \varepsilon_{i,j}\mid 1\leq i\neq j\leq n\rbrace&\\
    &&\Delta_+&=\lbrace\varepsilon_{i,j}\mid 1\leq i<j\leq n\rbrace&\\
    &&\Delta_s&=\lbrace \alpha_i\mid 1\leq i\leq n-1\rbrace&\\
    &&\Delta_s^{\vee}&=\lbrace \alpha_i^\vee \mid 1\leq i\leq n-1\rbrace.&
\end{align*}
We will denote by $Q_n$ the root lattice of $\sli_n$ and the set of fundamental weights corresponding to our choice of simple roots by
\begin{align*}
    &&\Omega&=\lbrace \omega_i\mid 1\leq i\leq n-1\rbrace\subset\mathfrak{h}^*.&
\end{align*}
\subsection{Vertex algebras}
We refer the reader to \cite{frenkelVertexAlgebrasAlgebraic2004} for an introduction to the theory of vertex algebras. Given a vertex algebra $V$, we will denote by $\ket{0}$ its vacuum vector and by $T$ its translation operator. For $a\in V$, we use the following notation for the modes of the corresponding field $a(z)$:
\begin{align*}
    &&a(z)&=\sum_{j\in\Z} a_{(j)} z^{-j-1}.&
\end{align*}
We also follow the convention where a vertex \textit{operator} algebra, or VOA, specifically means a vertex algebra with a conformal structure.\\\\
For homogeneous elements 
\begin{align*}
    x=x^{j_1}_{(-k_1)}\dots x^{j_m}_{(-k_m)}\ket{0}\in V&
\end{align*}
we will refer to $\sum_{i=1}^m k_i $ as the \textit{degree} of $x$.
\subsubsection{Symplectic boson \& free fermion vertex algebras}
The symplectic boson vertex algebra is a chiral analogue of the Weyl algebra of differential operators and is an important building block for many chiral constructions.
\begin{definition}
  The \textbf{symplectic boson} vertex algebra, denoted $\mathbf{Sb^{\otimes n}}$, is strongly generated by fields $\beta^i(z)$ and $\gamma^i(z)$ for $1\leq i\leq n$ that satisfy the operator product expansion (OPE):
  \begin{align*}
      &&\beta^i(z)\gamma^j(w)&\sim \frac{\delta_{i,j}}{z-w}.&
  \end{align*}
\end{definition}
\begin{remark}
The symplectic boson vertex algebra is sometimes also called the $\beta\gamma$-system. It is also known as the chiral differential operators (CDOs) of $\mathbb{C}^n$, $\mathcal{D}_{ch}(\mathbb{C}^n)$.
\end{remark}
An odd analogue of the symplectic boson VOA is given by the free fermion VOA.
\begin{definition}
 The \textbf{free fermion} vertex superalgebra, denoted $\mathbf{Ff^{\otimes n}}$, is strongly generated by odd fields $b^i(z)$ and $c^i(z)$ for $1\leq i\leq n$ that satisfy the OPE:
 \begin{align*}
     &&b^i(z)c^j(w)&\sim \frac{\delta_{i,j}}{z-w}.&
 \end{align*}
\end{definition}
\subsubsection{Affine vertex algebras}
Given a finite-dimensional Lie superalgebra $\g$ we consider the following induced representation of its affinization $\widehat{g}_k$
\begin{align*}
    &&V^k(\g)&=\mathrm{Ind}_{\g[\![t]\!]\oplus\C K}^{\widehat{g}} \C_k&
\end{align*}
where $\C_k$ is the one-dimensional representation of $\g[\![t]\!]\oplus\C K$ on which $K$ acts as $k\in\C$. 
\begin{definition}
 Let $\lbrace J^a\rbrace_{i=1}^{\dim\g}$ be a basis for $\g$. Then the \textbf{universal affine VOA} is the vertex (super)algebra structure on $V^k(\g)$ strongly generated by fields $J^i(z)$ satisfying the OPE
 \begin{align*}
     &&J^i(z)J^j(w)&\sim\frac{[J^i,J^j](w)}{z-w}+\frac{k(J^i,J^j)}{(z-w)^2}&
 \end{align*}
 where $(\cdot,\cdot)$ is a non-degenerate bilinear form on $\g$.
\end{definition}
If the level $k$ is not critical (i.e. not the dual coxeter number of $\g$), then $V^k(\g)$ is a VOA. Furthermore, as a $\widehat{\g}$-representation, $V^k(\g)$ has a unique simple quotient, denoted $L_k(\g)$, which also has the structure of a VOA.
\begin{definition}
 The \textbf{simple affine VOA}, denoted $L_k(\g)$, is the vertex (super)algebra structure inherited by the simple quotient of $V^k(\g)$. 
\end{definition}
We can also consider more general highest weight modules and their simple quotients. Given a $\g$-weight $\lambda$, we denote the Verma module for $\widehat{\g}$ of highest weight $\lambda$ by $M(\g,\lambda)$ and we will denote its simple quotient by $L(\g,\lambda)$. The following definition is crucial to understanding the submodule structures of these modules:
\begin{definition}
 A vector $v\in V^k(\g)$ is called \textbf{singular} if for all $x\in\widehat{\mathfrak{n}}_+$
 \begin{align*}
     &&x\cdot v&=0.&
 \end{align*}
\end{definition}
\subsubsection{$V^1(\psl_{n|n})$}\label{pslsection}
In this subsection we review background relating to the Lie superalgebra $\psl_{n|n}$ and its VOA, which is the primary object of study for this paper.
\begin{definition}
For $n\geq 2$, the projective special linear algebra \textbf{$\mathfrak{psl}_{n|n}$} is the simple Lie superalgebra of type $A(n-1,n-1)$. This algebra can be described by the set of matrices $X\in \mathrm{End}(\C^{n|n})$ such that
\begin{align*}
    &&\mathrm{str} X&=0,&
\end{align*}
quotiented by the ideal generated by the identity matrix.
\end{definition}
To highlight the subalgebra structure of this Lie algebra, we write a representative of $X\in \psl_{n|n}$ as
\begin{align*}
 X&=\left(\begin{array}{ccc|ccc}
         X_{1,1} & \dots & X_{1,n}  & X_{1,n+1}&\dots & X_{1,2n}  \\
         \vdots &\ddots & \vdots & \vdots &\ddots & \vdots \\
         X_{n,1} & \dots & X_{n,n}  & X_{n,n+1}&\dots & X_{n,2n} \\ \hline 
         X_{n+1,1} & \dots & X_{n+1,n} & X_{n+1,n+1}  &\dots & X_{n+1,2n} \\
         \vdots & \ddots & \vdots &\vdots & \ddots & \vdots \\
         X_{2n,1} & \dots & X_{2n,n} & X_{2n,n+1}  &\dots & X_{2n,2n}
    \end{array}\right).
\end{align*}
The two diagonal blocks form commuting $\sli_n$ subalgebras in $\psl_{n|n}$. In $V^1(\psl_{n|n})$, the generators related to the bottom-right block form a $V^{-1}(\sli_n)$-subalgebra, while the ones related to the top-left block form a $V^1(\sli_n)$-subalgebra. The off-diagonal blocks are odd.

Let now $E^{i,j}$ denote the elementary supermatrix acting on $\mathbb{C}^{n|n}$ with a 1 in the $(i,j)$\textsuperscript{th} entry and zeroes elsewhere. Following the conventions set in \cite{frappatDictionaryLieSuperalgebras1996}, we choose simple roots analogously to those chosen for $\sli_n$ in Section \ref{notation}:
\begin{align*}
    &&\Delta_s'&=\lbrace \alpha_i'\mid 1\leq i\leq 2n-1\rbrace.&
\end{align*}
Note that each of the $\alpha_i'$ are even except for $\alpha_n'$, which is odd.

The Chevalley generators corresponding to the root data are given by
\begin{align*}
    &&e_{i,j}&=E^{i,j}, &f_{i,j}&=E^{j,i}&
\end{align*}
for $1\leq i< j\leq 2n$ and
\begin{align*}
    &&h_i&=E^{i,i}-E^{i+1,i+1}&
\end{align*}
for $1\leq i\leq n-1$ and $n+1\leq i\leq 2n-1$. These generators span $\psl_{n|n}$. We will generically denote elements of this set of generators by $E^\alpha$, where $\alpha$ represents the appropriate single or double index. The strong generators of $V^1(\mathfrak{psl}_{n|n})$ corresponding to these Chevalley generators $E^\alpha(z)$ satisfy the OPE:
\begin{align*}
    &&E^\alpha(z)E^\beta(w)&\sim \frac{[E^\alpha,E^\beta](w)}{z-w}+\frac{\mathrm{str}(E^\alpha E^\beta)\mathbf{1}}{(z-w)^2}.&
\end{align*}
For convenience, we introduce the following notation where for $1\leq j< i\leq 2n$:
\begin{align*}
    &&D^{i,j}&=[E^{i,j},E^{j,i}]=\begin{cases}
   \frac{1}{n}\sum\limits_{\substack{ k=1 \\ k\neq j}}^n H^{j,k}+\frac{1}{n}\sum\limits_{\substack{ k=n+1 \\ k\neq i}}^{2n} H^{i,k}, & \text{if}\; j\leq n\; \text{and}\; i> n \\
        H^{i,j} & \text{if}\; i\leq n \\
        H^{j,i} & \text{if}\; j>n
    \end{cases}
\end{align*}
and
\begin{align*}
        &&H^{i,j}&=E^{i,i}-E^{j,j}.&
\end{align*}
\subsubsection{Heisenberg \& Lattice vertex algebras \& $L_1(\mathfrak{sl}_n)$}\label{latticesection}
Let $\mathfrak{H}$ denote the level 1 Heisenberg Lie algebra with basis $\lbrace J_{(j)}\rbrace_{j\in\Z}$ and let $\mathfrak{H}_+\subset\mathfrak{H}$ denote the subalgebra generated by $\lbrace J_{(j)}\rbrace_{j\in\mathbb{N}}$. $\mathfrak{H}$ has a vacuum representation
\begin{align*}
    &&\mathcal{H}&=\mathrm{Ind}_{\mathfrak{H}_+}^{\mathfrak{H}}\C&
\end{align*}
where $\C$ indicates the 1-dimensional representation of $\mathfrak{H}_+$. We can generalise this construction to higher rank cases as well. Let $L$ be a lattice of rank $r$ with basis
\begin{align*}
    &&L &=\Z\phi^1\oplus\dots\oplus\Z\phi^r&
\end{align*}
and equipped with a symmetric bilinear form
\begin{align*}
    &&B_L:L\times L&\to \C.&
\end{align*}
From this data we define the rank $r$ Heisenberg algebra
\begin{align*}
    &&\mathfrak{h}^L&=\C\otimes_\Z L&
\end{align*}
and consider its central extension of with respect to $B_L$:
\begin{align*}
    0\to\C\to\mathfrak{H}^L\to \mathfrak{h}^L(\!(t)\!)\to 0.
\end{align*}
This extended algebra has a vacuum representation
\begin{align*}
    &&\mathcal{H}^L&=\mathrm{Ind}_{\mathfrak{H}^L_+}^{\mathfrak{H}^L}\C&
\end{align*}
where as before $\mathfrak{H}^L_+\subset\mathfrak{H}^L$ is the subalgebra of non-negative modes and $\C$ denotes the 1-dimensional representation of $\mathfrak{H}^L_+$.

\begin{definition}
The rank $r$ \textbf{Heisenberg VOA} associated to the lattice $L$ is the VOA structure on $\mathcal{H}^L$ strongly generated by fields $J^{\phi , i }(z)$ for $1\leq i\leq r$ and satisfying the OPE
\begin{align*}
    &&J^{\phi,i}(z)J^{\phi,j}(w)&\sim \frac{B_L(\phi^i,\phi^j)}{(z-w)^2}.&
\end{align*}
\end{definition}
All irreducible representations of $\mathcal{H}^L$ are in the following class of modules.
\begin{definition}
Given $\mu\in \mathfrak{h}^L$ the \textbf{Fock module}, denoted $F^\phi_\mu$, for $\mathcal{H}^L$ of highest weight $\mu$ is given by 
\begin{align*}
    &&F^L_\mu&=\mathrm{Ind}_{\mathfrak{H}^L_+}^{\mathfrak{H}^L}\C_\mu&
\end{align*}
where $\C_\mu$ is the representation of $\mathfrak{H}^L_+$ such that $J^{\phi,i}_{(0)}$ acts by $B_L( \mu, \phi^i)=\mu_i$ and $J^{\phi,i}_{(m)}$ acts by zero for all $m>0$.
\end{definition}
If the basis $\lbrace \phi^1,\dots,\phi^r\rbrace$ of $L$ is orthogonal, one can produce a decomposition
\begin{align*}
    &&\mathcal{H}^L&=\bigotimes_{i=1}^r \mathcal{H}^{\phi^i},& F^L_\mu&=\bigotimes_{i=1}^r F^{\phi^i}_{\mu_i}.&
\end{align*}
With these modules, we can define another important class of vertex algebras. Set
\begin{align*}
    &&V^L&=\bigoplus_{\mu\in L} F_\mu^L.&
\end{align*}
\begin{definition}
The \textbf{lattice VOA} associated to an integral lattice $L$ is the vertex (super)algebra structure on $V^L$ strongly generated by fields $J^{\phi , i}(z)$ and $\exp(\phi^i)(z)$ for $1\leq i\leq r$ which satisfy
  \begin{align*}
      &&J^{\phi , i}(z)J^{\phi , j}(w)&\sim \frac{B_L(\phi^i,\phi^j)}{(z-w)^2},& \exp(\phi^i)(z)\exp(\phi^j)(w)&\sim \frac{:\exp(\phi^i+\phi^j):(w)}{(z-w)^{B_L(\phi^i,\phi^j)}}&
\end{align*}
and
\begin{align*}
    &&T\exp(\phi^i)(z)&=:\phi^i\exp(\phi^i):(z).&
  \end{align*}
\end{definition}
Note that if $L$ is even, then $V^L$ is a vertex algebra.\\\\
Let $L^\vee=\lbrace \mu\in L\otimes\mathbb{Q}\mid B_L(\mu,\nu)\in\Z,\;\forall \nu\in L\rbrace$ be the dual lattice of $L$. Then for each $\lambda\in L^\vee/L$ the space
\begin{align*}
    &&V^L_\lambda&=\bigoplus_{\mu\in\lambda+L} F^L_\mu&
\end{align*}
has the structure of a $V^L$-module. If $L$ is positive definite, then these are simple modules and all simple $V^L$ -modules are of this form \cite{frenkelVertexAlgebrasAlgebraic2004}.\\\\
With this established, we can now state some results that will be employed to identify the BRST cohomology used to define the 3d A-model boundary VOA of relevance in this paper:
\begin{theorem}[\mbox{\cite{frenkelBasicRepresentationsAffine1980,segalUnitaryRepresentationsInfinite1981}}]\label{positivedecomp}
 If $\g$ is a simply laced Lie algebra then there are VOA isomorphisms
 \begin{align*}
     L_1(\g)\cong V^{L},&
 \end{align*}
 where $L$ is the root lattice for $\g$.
\end{theorem}
We will be concerned with the case where $\g=\sli_n$ and $L=Q_n$. If we let $Q_n^\vee$ denote the dual lattice to $Q_n$ and note that $Q_n^\vee/Q_n\cong \Z_{n}$ we can describe the simple modules for $L_1(\sli_n)$ as follows.
\begin{theorem}[\mbox{\cite{dongVertexAlgebrasAssociated1993}}]\label{level1simples}
 The simple representations of $V^{Q_n}\cong L_1(\sli_n)$ are given by either
 \begin{align*}
     &&L_1(\sli_n,\omega_j)&\cong \bigoplus_{\mu\in \omega_j+Q_n} F^{Q_n}_\mu&
 \end{align*}
 where $\omega_j$ for $1\leq j\leq n-1$ are fundamental weights for $\sli_n$, or $L_1(\sli_n,0)=L_1(\sli_n)$.
\end{theorem}
\subsubsection{Singlet vertex algebra \& $L_{-1}(\mathfrak{sl}_n)$}\label{singletsection}
One final vertex algebra we will need is the $p=2$ \textbf{singlet algebra}, which we denote by $M$ and is also known as the $W_3$-algebra with central charge $c=-2$. We will not explicitly define it here, however the interested reader can refer to Section 2.3 of \cite{ballin3dMirrorSymmetry2023} for a review. In \cite{wangClassificationIrreducibleModules1998} all irreducible modules for the singlet algebra were classified and it was shown that given $\mu\in\C$ there exists an irreducible highest weight $M$-module which we will denote $M_\mu$.\\\\

The modules $M_\mu$ of the singlet algebra as well as lattice vertex algebras can be used to realise $L_{-1}(\mathfrak{sl}_n)$ as well as some of its modules. Let
\begin{align*}
    &&P(m)&=\lbrace \lambda=(\lambda_1,\dots,\lambda_n)\in\Z^n\mid \sum_{i=1}^n \lambda_i=m\rbrace.&
\end{align*}
Define a lattice
\begin{equation*}
    L_\varphi= \Z\varphi^1\oplus\dots\oplus\Z\varphi^n, \quad B_\varphi(\varphi^i,\varphi^j)=-\delta_{i,j},
\end{equation*}
and let $\mathcal{H}^{\varphi}$ be the respective rank-$r$ Heisenberg VOA as defined above. Set
\begin{align*}
    &&L_c=&-\mathbb{Z}(\varphi_1+\dots+\varphi_n)&
\end{align*}
and with $\Tilde{\varphi}_i = \varphi_i - \varphi_{i+1}$, $1\leq i < n$ define
\begin{equation*}
    L_{\tilde{\varphi}} = \mathbb{Z}\Tilde{\varphi}_{1} \oplus \dots \oplus \mathbb{Z}\Tilde{\varphi}_{n-1}. 
\end{equation*}
Note that if we identify $L_\varphi$ with the lattice associated with the Cartan subalgebra of $\mathfrak{gl}_n$, $L_{\tilde{\varphi}}$ corresponds to the lattice for the Cartan of the $\mathfrak{sl}_n$ subalgebra.\footnote{By Cartan subalgebra, we mean in this case the lattices above equipped with the negative of the Cartan matrix as a bilinear form.}
\begin{proposition}[\mbox{\cite{adamovicVertexAlgebrasRelated2018}}]\label{decompiso}
There is an isomorphism of VOAs
\begin{align*}
    &&\mathcal{H}^{{\varphi}}&\cong \mathcal{H}^{{\Tilde{\varphi}}}\otimes\mathcal{H}^{c}.&
\end{align*}
\end{proposition}
Given a weight $\lambda\in P(m)$, we can perform an orthogonal decomposition with respect to this isomorphism. In particular the corresponding $\mathcal{H}^{c}$ weight is given by
\begin{align*}
    &&\lambda_0&=\frac{1}{n}(\lambda_1+\dots+\lambda_n)&
\end{align*}
and the corresponding $\mathcal{H}^{{\tilde{\varphi}}}$ weight is given by\footnote{If we replace $\lambda_i-\lambda_{i+1}$ with the simple roots of $\sli_n$, this formula corresponds to a change of basis formula from simple roots to fundamental weights \cite[Table 1]{humphreysIntroductionLieAlgebras1972}.}
\begin{align*}
    &&\lambda_k^\vee=\alpha_k^\vee(\lambda)=\frac{1}{n}\sum_{i=1}^k i(n-k)&(\lambda_i-\lambda_{i+1})+\frac{1}{n}\sum_{i=k+1}^{n-1} k(n-i)(\lambda_i-\lambda_{i+1})&
\end{align*}
for $1\leq k\leq n-1$. Given this, we reformulate \cite[Theorem 4.4]{adamovicVertexAlgebrasRelated2018} as follows
\begin{theorem}[\mbox{{\cite{adamovicVertexAlgebrasRelated2018}}}]\label{negativedecomp}
There is an isomorphism of VOAs
\begin{align*}
    &&L_{-1}(\sli_n)&\cong \bigoplus_{\lambda\in P(0)} M_{\lambda_1}\otimes\dots\otimes M_{\lambda_n}\otimes F^{\tilde{\varphi}}_{\lambda^\vee},&
\end{align*}
where $F_{\lambda^\vee}^{\Tilde{\varphi}}$ is $F_\lambda^{\varphi}$ viewed as a module for $\mathcal{H}^{L_{\Tilde{\varphi}}}$. Similarly we have
\begin{align*}
    &&L(\sli_n,\lambda(m))&\cong \bigoplus_{\lambda\in P(m)} M_{\lambda_1}\otimes\dots\otimes M_{\lambda_n}\otimes F^{\tilde{\varphi}}_{\lambda^\vee}&
\end{align*}
where
\begin{align*}
    &&\lambda(m)&=\begin{cases}
        m\omega_1, & j\geq 0 \\
        m\omega_{n-1}, & j<0
    \end{cases}&
\end{align*}
and all irreducible  ordinary $L_{-1}(\sli_n)$ modules are of the form $L(\sli_n,\lambda(m))$ for some $m\in\Z$.
\end{theorem}
\subsubsection{Associated varieties}
The associated variety of a vertex algebra $V$ was first defined in \cite{arakawaRemarkC_2Cofiniteness2010} and provides information about the structure and representation theory of the vertex algebra. Aspects of this construction were later generalised to the super setting in other work including \cite{liRemarksAssociatedVarieties2020}. As mentioned in the introduction, the associated varieties of VOAs appearing in certain 3d A-models and 4d quantum field theories are conjectured to be, and in some cases proven to be \cite{arakawaChiralAlgebrasClass2018}, the Higgs branches of those theories. The motivation for this paper is to prove such a conjecture for the 3d A-model of SQED[$n$].
\begin{definition}
Let $V$ be a strongly finitely generated vertex algebra with strong generators $\lbrace a^i,\dots,a^r\rbrace$ and consider the subspace
\begin{align*}
    &&C_2(V)&=\mathrm{span}_\C\lbrace a^i_{(-n)}v\mid 1\leq i\leq r,\; n\geq 2,\; v\in V\rbrace.&
\end{align*}
The \textbf{Zhu} $\mathbf{C_2}$\textbf{-algebra} of $V$ is defined to be
\begin{align*}
    &&R_V&= V/C_2(V)&
\end{align*}
and has a commutative product given by
\begin{align*}
    &&\overline{a}\cdot \overline{b}&=\overline{a_{(-1)}b},& a,b\in V.&
\end{align*}
The \textbf{associated variety} of $V$ is defined as
\begin{align*}
    &&X_V&=\mathrm{specm} R_V .&
\end{align*}
\end{definition}
The Zhu $C_2$-algebra $R_V$ is finitely generated if and only if $V$ is strongly finitely generated and has the structure of a Poisson algebra whose bracket is given by
\begin{align*}
    &&\lbrace\cdot,\cdot\rbrace:R_V\times R_V&\to R_V&\\
    &&(\overline{x},\overline{y})&\mapsto \overline{x_{(0)}y}.&
\end{align*}
It follows that the associated variety of a strongly finitely generated vertex algebra is a Poisson variety of finite type. This construction also gives a good way of characterising how ``large" a VOA is:
\begin{definition}[\cite{zhuModularInvarianceCharacters1996},\cite{arakawaQuasilisseVertexAlgebras2017}]
A vertex algebra $V$ is called $\mathbf{C_2}$\textbf{-cofinite} or \textbf{lisse} if $C_2(V)<\infty$, or equivalently if $\dim X_V=0$. We say $V$ is \textbf{quasi-lisse} if $X_V$ has finitely many symplectic leaves.
\end{definition}
The lisse condition is widely used throughout the study of VOAs and has many important consequences. For instance, lisse VOAs have finitely many irreducible representations and their characters satisfy a modular invariance property \cite{zhuModularInvarianceCharacters1996}. The quasi-lisse condition is a weaker version of this property and is of interest since VOAs constructed from supersymmetric QFTs, such as those appearing in this paper, are typically expected to satisfy this condition. Quasi-lisse VOAs also enjoy certain properties, for example their characters are known to satisfy modular linear differential equations \cite{arakawaQuasilisseVertexAlgebras2017}.

\begin{example}\label{ex:affineassoc}
Consider the associated variety of the universal affine Kac-Moody VOA $V^k(\g)$. Following \cite{arakawaSheetsAssociatedVarieties2019,liRemarksAssociatedVarieties2020}, the Zhu $C_2$-algebra of this VOA is given by
\begin{align*}
&&C_2(V^k(\g))&=t^{-2}\g[t^{-1}]V^k(\g)&
\end{align*}
so that 
\begin{align*}
&&R_{V^k(\g)}&=\mathbb{C}[\g^*].&
\end{align*}
Since expressions involving odd elements are nilpotent, they should not contribute to the associated variety. To emphasise this we consider the reduced algebra and variety
\begin{align*}
    &&(R_{V^k(\g)})_{\mathrm{red}}&=\mathbb{C}[\mathfrak{g}_{\bar{0}}^*].&\\
    &&X_{V^k(\g)}&=\g_{\bar{0}}^*.&
\end{align*}
Consider now a quotient $V=V^k(\g)/U$ by some submodule $U$. We then have a surjective algebra morphism
\begin{align*}
    &&\pi_V: \mathbb{C}[\g^*]&\twoheadrightarrow V/(t^{-2}\g[t^{-1}]V)&\\
    &&x&\mapsto \overline{x_{(-1)}}+t^{-2}\g[t^{-1}]V.&
\end{align*}
The associated variety of $V$ is the zero locus of $(\ker\pi_V)_{\mathrm{red}}$, which in turn implies that $X_V$ is a closed $G_{\bar{0}}$-invariant subvariety of $\g_{\bar{0}}^*$. We will make use the following lemma.
\begin{lemma}\label{lem:submodule}
Let $U^{'} \subseteq U$ be two sumbodules of $V^{k}(\g)$. Then for $V=V^{k}(\g)/U$ and $V^{'}=V^{k}(\g)/{U^{'}}$ we have $\ker\pi_{V'} \subseteq \ker \pi_{V}$. Denoting the quotient map to the reduced algebra
\begin{align*}
    &&\sigma: R_\V&\twoheadrightarrow (R_\V)_{\mathrm{red}}&
\end{align*}
and setting $I_V= \sigma(\ker_V)$, we have $I_{V'}\subseteq I_{V}$.
 \end{lemma}
 \begin{proof}
The first statement follows directly from the definition of the map $\pi_V$, while the second follows directly from the first.
\end{proof}
In the following sections, we will recall results on varieties that appear as the associated varieties for $L_k(\sli_n)$ at certain levels.
\end{example}

\subsection{Sheets \& nilpotent orbits of Type A} 

\subsubsection{Nilpotent orbits of type A}\label{subsec:nil}  For a brief general introduction to the theory of nilpotent orbits we suggest consulting \cite{hendersonSingularitiesNilpotentOrbit2015}.  Briefly put, given a simple Lie algebra $\g$, one can define its nilpotent cone $\mathcal{N}(\g)$ which is a conic algebraic variety. The variety $\mathcal{N}(\g)$ is the disjoint union of finitely many $G$-orbits which are irreducible if $G$ is connected. These orbits have a partial ordering and the closure of a nilpotent orbit is given as the union of itself with all nilpotent orbits below it in the partial ordering. For the purposes of this paper, we are only interested in minimal non-zero nilpotent orbits of type A and their closures, which we now define explicitly.\\

\begin{definition}
The \textbf{nilpotent cone} of type A is
\begin{align*}
    &&\mathcal{N}(\sli_n)&=\lbrace Z\in\sli_n\mid \text{all eigenvalues of $Z$ are 0} \rbrace.
\end{align*}
The \textbf{minimal nilpotent orbit} of type A is
\begin{align*}
    &&\Omin(\sli_n)&=\lbrace Z\in\sli_n\mid \mathrm{rank}(Z)=1\rbrace.&
\end{align*}
\end{definition}
The closure of $\Omin(\sli_n)$ is found by taking its union with the trivial orbit:
\begin{align*}
    &&\overline{\Omin(\sli_n)}&=\lbrace Z\in\sli_n\mid \rank(Z)\leq 1\rbrace.&
\end{align*}
This condition is equivalent to requiring that the $2\times 2$ minors of any matrix in $\overline{\Omin(\sli_n)}$ vanish. This allows us to express the coordinate ring of $\overline{\Omin(\sli_n)}$ as follows:
\begin{align*}
    &&I=\lbrace Z \mapsto Z_{i,j}Z_{k,l}-Z_{i,l} Z_{k,j}\mid& Z\in\sli_n ,\;i<k,\;  j<l\rbrace\subset \mathbb{C}[\sli_n] &\\
    &\implies&\overline{\Omin(\sli_n)}&=\mathrm{Spec}(\mathbb{C}[\sli_n]/I).&
\end{align*}

\begin{remark}\label{rem:partition}
    More generally, nilpotent orbits of $\sli_n$ are labelled by partitions $\alpha$ of $n$ that encode the size of the Jordan blocks of their elements. In the case of the minimal nilpotent orbit $\alpha = (2,1,\dots ,1)$, which implies that its elements are as given in the above definition.
\end{remark}

\begin{remark}\label{rem:iden}
    As we have just discussed in Example~\ref{ex:affineassoc}, in the context of associated varieties of affine Kac-Moody algebras it is more natural to obtain $G$-invariant subvarieties of $\mathfrak{g}^*$ instead of $\g$. If we identify $\sli_n$ with $\sli_n^*$ with the bilinear form $(W,Z)= \frac{1}{2n}\trace (WZ)$, $\overline{\Omin (\sli_n)}$ can be expressed as $ \mathrm{Spec} (\mathrm{Sym}(\sli_n)/J)$ where $\mathrm{Sym}(\sli_n) \cong \mathbb{C}[\sli_n^*]$ and $ J =  \lbrace Z_{i,j}Z_{k,l}-Z_{i,l} Z_{k,j}\mid Z\in\sli_n ,\;i<k,\;  j<l\rbrace $.
\end{remark}

\subsubsection{Sheets of type A}\label{subsubsec:sheetsA} In order to connect our discussion to previous results in the literature, we need to introduce sheets. In the same way nilpotent orbits are irreducible subvarieties of the nilpotent cone, a sheet of a Lie algebra is an irreducible subvariety of the subset
\begin{align*}
    &&\g^{(m)}&=\lbrace x\in\g\mid \dim \g^x=m\rbrace,& m\in\mathbb{N},&
\end{align*}
where $\g^x$ denotes the centraliser of $x$ in $\g$. A key difference between sheets and nilpotent orbits is that sheets need not be contained in $\mathcal{N}(\g)$. For an introduction to the topic, we suggest consulting \cite[\S 39]{tauvelLieAlgebrasAlgebraic2005}.\\\\
Sheets of Lie algebras are closed, irreducible and $G$-invariant varieties that contain a unique nilpotent orbit. Since for our purposes we are only concerned with the closure of sheets of type A containing the minimal nilpotent orbit, we make use of the following result:
\begin{lemma}{\cite{arakawaSheetsAssociatedVarieties2019}}\label{sheetlemma}
Let $\levi\subseteq\sli_n$ be a Levi subalgebra such that its centraliser is one-dimensional:
\begin{align*}
    &&\mathfrak{z}(\levi)&=\C\xi,& \xi\in\mathfrak{h}\backslash\lbrace 0\rbrace.&
\end{align*}
Then the closure of the sheet associated to the Jordan class of $\levi$ is given by
 \begin{align*}
     &&\overline{\mathbb{S}_\levi(\sli_n)}&=SL_n\cdot(\C\xi+\mathfrak{p}_u)&
 \end{align*}
 where
 \begin{align*}
     &&\mathfrak{p}&=\levi\oplus\mathfrak{p}_u&
 \end{align*}
 is a parabolic subalgebra and $\mathfrak{p}_u$ is its nilradical.
\end{lemma}
We are interested in the case $\overline{\Omin(\sli_n)}\subset \overline{\Sh(\sli_n)}$. As mentioned in Remark~\ref{rem:partition}, the minimal nilpotent orbit of type A corresponds to the partition $\alpha=(2,1,\dots,1)$, and the relevant Levi subalgebra is determined by the conjugate partition $\overline{\alpha}=(n-1,1)$ \cite{arakawaSheetsAssociatedVarieties2019}. For completeness, we give an explicit description of elements of the sheet $\overline{\Sh(\sli_n)}$ in Appendix~\ref{app:B}. It is possible to distinguish between $2\times 2$ minors that vanish on the both $\overline{\Omin(\sli_n)}$ and $\overline{\Sh(\sli_n)}$, and those that only vanish on $\overline{\Omin(\sli_n)}$. We also do this in Appendix~\ref{app:B}.

\section{BRST cohomology \& 3d A-model boundary VOA}
In this section we will prove the statement made in \cite{beemFreeFieldRealisation2023} that the boundary VOA for the 3d A-model of SQED with $n\geq 3$ hypermultiplets is indeed $L_1(\psl_{n|n})$. The proof of this fact follows rather directly from the results of \cite{ballin3dMirrorSymmetry2023,creutzigTensorCategoriesAffine2021,adamovicConformalEmbeddingsAffine2019}.
We will briefly outline the data used to define the 3d A-model and boundary VOA for abelian gauge theories, but the interested reader may refer to \cite{costelloHiggsCoulombBranches2018} for more details on the physical construction.\\\\
The 3d A-model of an abelian gauge theory is determined by a gauge group $G=U(1)^r$ and a representation on $V=T^*N$ where $N \cong \mathbb{C}^n$. The weights of the $G$-action on $N$, which define the representation, can be put in a $n \times r $ matrix that we denote by $\rho$. The matrix $\rho$ is the all that is needed to define the boundary VOA of an abelian gauge theory, and to ease the formulation of the definition we take it to be unimodular. For SQED[$n$], the case of interest in this paper, we have $V=T^*\mathbb{C}^n$, $r=1$ and
\begin{align*}
    &&\rho&=\left(\begin{array}{c}
         1 \\
         \vdots \\
         1
    \end{array} \right),&
\end{align*}
which means that $U(1)$ acts with unit weight on each copy of $\mathbb{C}$.\\\\

Consider the VOA $\mathrm{Sb}^{\otimes n}\otimes\mathrm{Ff}^{\otimes n}$. We can define $r$ currents
\begin{equation}
    J^{\rho, i} = \sum_{j=1}^n \rho_{i}^j (:\gamma^i \beta^i: + :b^i c^i: ). 
\end{equation}
These generate a level-zero $\widehat{\mathfrak{gl}_1^r}$ sub-VOA of $\mathrm{Sb}^{\otimes n}\otimes\mathrm{Ff}^{\otimes n}$. In particular, for SQED[$n$] the $\widehat{\mathfrak{gl}_1}$ sub-VOA is generated by the current
\begin{equation}
    J^{\rho} = \sum_{j=1}^n  (:\gamma^i \beta^i: + :b^i c^i: ). 
\end{equation}

\begin{definition}\label{def:VA}
The \textbf{boundary VOA} for the 3d A-model of an abelian gauge theory with weight matrix $\rho$ is defined as the following relative BRST cohomology~\cite{voronov1993semi}
\begin{align*}
    &&V^A_\rho&=H^{\frac{\infty}{2}+\bullet}(\mathfrak{gl}_1^r,\widehat{\mathfrak{gl}_1^r},\mathrm{Sb}^{\otimes n}\otimes\mathrm{Ff}^{\otimes n}),&&
\end{align*}
where the level-0 Heisenberg VOA $\widehat{\mathfrak{gl}_1^r}$ is defined as above.
\end{definition}
\vspace{0.2pt}

\begin{remark}
It is worth mentioning that physically the free fermions are included to cancel an anomaly. Mathematically, we can think of this VOA as being the global sections of the sheaf of chiral differential operators on the symplectic reduction of $V$ with respect to $G$. It was shown in \cite{gorbounovGerbesChiralDifferential1999} that there is often a topological obstruction to the construction of a sheaf of chiral differential operators corresponding to the gauge anomaly.\\\\
Other anomaly cancellations may be used, and this was done for the example in \cite{kuwabaraVertexAlgebrasAssociated2017}. Here we use the anomaly cancellation data preferred by physicists. As mentioned in the introduction, this is expected to lead to quasi-lisse VOAs that moreover manifest outer-automorphisms of physical significance \cite{beemFreeFieldRealisation2023}.
\end{remark}
It was shown in \cite{beemFreeFieldRealisation2023,ballin3dMirrorSymmetry2023} that for abelian gauge theories of this type, the cohomology defining $V^A_\rho$ is concentrated in cohomological degree 0. In \cite{ballin3dMirrorSymmetry2023} a decomposition for $V_\rho^A$ is given in terms of singlet algebra and Fock modules resulting from a lattice decomposition determined by $\rho$. Here we will look at this decomposition for the case of interest in this paper.
\begin{proposition}\label{cats}
For $r=1$ and
\begin{align*}
    &&\rho&=\left(\begin{array}{c}
         1 \\
         \vdots \\
         1
    \end{array} \right)&
\end{align*}
we have an isomorphism of vertex operator superalgebras
\begin{align*}
    &&V_\rho^A&\cong \bigoplus_{j=0}^{n-1}\bigoplus_{m\in\Z} \bigoplus_{\substack{ \mu\in \omega_j+Q_n,\\ \lambda\in P(-mn-j)}}  M_{\lambda_1}\otimes\dots\otimes M_{\lambda_n}\otimes F^{\tilde{\varphi}}_{\lambda^\vee}\otimes F^{Q_n}_\mu&
\end{align*}
where the expressions appearing on the RHS were introduced in Sections \ref{latticesection} and \ref{singletsection}.
\end{proposition}
\begin{proof}
To prove this, we use Proposition 3.2 of \cite{ballin3dMirrorSymmetry2023}. The $n\times (n-1)$ matrix $\rho^\vee$ of that Proposition (the unimodular matrix associated to the mirror theory) is in our case
\begin{align*}
    &&\rho^\vee&=\left(\begin{array}{cccc}
        1 &  &  &   \\
        -1 & 1 &  &     \\
         & -1 & \ddots     & \\
       &  & \ddots    & 1 \\
        &  &   & -1
    \end{array}\right).&
\end{align*}
These matrices provide an orthogonal decomposition
\begin{align*}
    &&\C^n&=\rho(\C)\oplus\rho^\vee(\C^{n-1}).&
\end{align*}
The associated orthogonal projections are given by
\begin{align*}
    &&p:\C^n&\twoheadrightarrow \rho(\C)&\\
    &&\left(\begin{array}{c}
         v_1  \\
         \vdots \\
         v_n
    \end{array}\right)&\mapsto \frac{1}{n}\left(\begin{array}{c}
         v_1+\dots+v_n \\
         \vdots \\
         v_1+\dots+v_n
    \end{array}\right)&\\
    &&p^\perp:\C^n&\twoheadrightarrow \rho^\vee(\C^{n-1})&\\
    &&\left(\begin{array}{c}
         v_1  \\
         \vdots \\
         v_n
    \end{array}\right)&\mapsto \frac{1}{n}\left(\begin{array}{c}
         (n-1)v_1-v_2-\dots -v_n  \\
          -v_1+(n-1)v_2-v_3-\dots -v_n \\
          \vdots \\
          -v_1-\dots -v_{n-1}+(n-1)v_n
    \end{array} \right).&
\end{align*}
Using this data, we can compute the abelian group $H$ appearing in the decomposition from Proposition 3.2 of \cite{ballin3dMirrorSymmetry2023}:
\begin{align*}
    &&H&=N/\Z=N^\vee/\Z^{n-1}=\Z_n&
\end{align*}
where
\begin{align*}
    &&N&=\rho^{-1}(p(\Z^n))=\frac{1}{n} \Z, &N^\vee&=(\rho^\vee)^{-1}(p^\perp(\Z^n))=\frac{1}{n}\Z^{n-1}.&
\end{align*}
In particular, given a weight $\lambda=(\lambda_1,\dots,\lambda_n)\in\Z^n$ we get
\begin{align*}
    &&\lambda&=\rho(\lambda_0)+\rho^\vee(\lambda^\vee),& \lambda_0\in N,\; \lambda^\vee\in N^\vee&
\end{align*}
such that $[\lambda_0]=[\lambda^\vee]\in H$ where
\begin{align*}
    &&\lambda_0&=\frac{1}{n}(\lambda_1+\dots+\lambda_n)&\\
    &&\lambda_k^\vee=\frac{1}{n}\sum_{i=1}^k i(n-k)&(\lambda_i-\lambda_{i+1})+\frac{1}{n}\sum_{i=k+1}^{n-1} k(n-i)(\lambda_i-\lambda_{i+1})&
\end{align*}
which are the same formulae we encountered in Section \ref{singletsection}. Since we must have $\lambda\in P(-mn-j)\subset\Z^n$ for some $m\in\Z$ and $0\leq j\leq n-1$ and
\begin{align*}
    &&[\lambda_0]&=\sum_{i=1}^n \lambda_i \mod n&
\end{align*}
we can identify $j$ with the class $[\lambda_0]=[\lambda^\vee]$.\\\\
Letting $C_n$ be the Cartan matrix of $\sli_n$, we then define three lattices
\begin{align*}
&&L_{\theta^\perp}&=\Z\theta^\perp_1\oplus\dots\oplus\Z\theta_{n-1}^\perp&\\
    &&L_{\eta^\perp}&=\Z\eta^\perp_1\oplus\dots\oplus\Z\eta^\perp_{n-1}&\\
    &&L_{\phi}&=\Z\phi_1\oplus\dots\oplus\Z\phi_n&
\end{align*}
with bilinear forms
\begin{align*}
    B_{\theta^\perp}(\theta^\perp_i,\theta^\perp_j)&=-B_{\eta^\perp}(\eta^\perp_i,\eta^\perp_j)=C_n\\
    &B_\phi(\phi_i,\phi_j)=\delta_{i,j}.
\end{align*}
Proposition 3.2 of \cite{ballin3dMirrorSymmetry2023} tells us that we have a VOA isomorphism
\begin{align*}
    &&V_\rho^A&\cong \bigoplus_{\substack{\lambda_0\in N, \; \lambda^\vee,\nu^\vee\in N^\vee \\ [\lambda_0]=[\lambda^\vee]=[\nu^\vee]}} M^\phi_{(\rho(\lambda_0)+\rho^\vee(\lambda^\vee))\cdot\phi}\otimes F^{\eta^\perp}_{\lambda^\vee}\otimes F^{\theta^\perp}_{\nu^\vee}&
\end{align*}
Observing that $L_{\theta^\perp}=Q_n$ and $L_{\eta^{\perp}}=L_{\tilde{\varphi}}$ we can conclude that
\begin{align*}
    &&V_\rho^A&\cong \bigoplus_{j=0}^{n-1}\bigoplus_{m\in\Z}\bigoplus_{\substack{ \mu\in \omega_j+Q_n,\\ \lambda\in P(-mn-j)}} M_{\lambda_1}\otimes\dots\otimes M_{\lambda_n}\otimes F^{\Tilde{\varphi}}_{\lambda^\vee}\otimes F^{Q_n}_\mu.&
\end{align*}
\end{proof}
\begin{theorem}
For all $n\geq 3$, $r=1$ and $\rho = (1,\ldots ,1)^T$ an $1 \times n$ matrix as above the 3d A-model boundary VOA is
\begin{align*}
    &&V_\rho^A&\cong L_1(\psl_{n|n}).&
\end{align*}   
\end{theorem}
\begin{proof}
It was shown in \cite{adamovicConformalEmbeddingsAffine2019,creutzigTensorCategoriesAffine2021} that $L_1(\psl_{n|n})$ is a simple current extension of $L_{-1}(\sli_n)\otimes L_1(\sli_n)$:
\begin{align*}
    &&L_1(\psl_{n|n})&\cong \bigoplus_{j=0}^{n-1} \bigoplus_{m\in\Z}  L_{-1}(\sli_n,\lambda(-mn-j))\otimes L_1(\sli_n,\omega_j) &
\end{align*}
Applying Theorems \ref{level1simples} and \ref{negativedecomp} gives
\begin{align*}
    &&L_1(\psl_{n|n})&\cong \bigoplus_{j=0}^{n-1}\bigoplus_{m\in\Z} \bigoplus_{\substack{ \mu\in \omega_j+Q_n,\\\lambda\in P(-mn-j)}} M_{\lambda_1}\otimes\dots M_{\lambda_n}\otimes F^{\Tilde{\varphi }}_{\lambda^\vee}\otimes F^{Q_n}_{\mu}&
\end{align*}
which is the expression for $V_\rho^A$ given in Proposition \ref{cats}.
\end{proof}
\begin{remark}
    In~\cite{beemFreeFieldRealisation2023}, the above free field realisation was further related to the geometry of $\overline{\Omin (\sli_n)}$ and its resolutions. This was achieved by representing the singlet VOA modules appearing above as fixed-charged sectors of symplectic fermions, and by bosonising those to arrive at fields that could be interpreted as CDOs on openly embedded subsets of these varieties.
\end{remark}
\begin{remark}\label{rem:mirror}
    In~\cite{ballin3dMirrorSymmetry2023}, in the context of 3d mirror symmetry, it was proven that $V^A_\rho$ is isomorphic to yet another VOA $V^{B}_{\rho^\vee}$, the B-model VOA arising from the (mirror) theory defined by $\rho^{\vee}$. $V^B_{\rho^\vee}$ is also not obviously isomorphic to $L_1(\psl_{n|n})$. We leave the investigation of mirror symmetry to future work.
\end{remark}
\section{Singular vectors of $V^1(\mathfrak{psl}_{n|n})$}

To compute the associated variety of $L_1(\psl_{n|n})$, we first need to consider the maximal submodule of $V^1(\psl_{n|n})$. One way of approaching this task is to identify singular vectors, as these generate submodules.
\begin{proposition}\label{singularvectorprop}
The vector
\begin{align*}
    &&\chi&=E^{1,2n-1}_{(-1)}E^{1,2n}_{(-1)}\ket{0}&
\end{align*}
is singular in $V^1(\mathfrak{psl}_{n|n})$. The submodule generated by this vector $U=\langle\chi\rangle$ is a proper non-zero submodule of $V^1(\psl_{n|n})$.
\end{proposition}
    \begin{proof}
To show that $\chi$ is singular, we must demonstrate that it is annihilated by $\widehat{\mathfrak{n}}_+$. Any non-negative degree mode will annihilate $\chi$ if it commutes with both of $E^{1,2n-1}_{(-1)}$ and $E^{1,2n}_{(-1)}$. Since the relevant commutation relation for these modes is
\begin{align*}
    &&[E^{i,j}_{(m)},E^{k,l}_{(-1)}]&=[E^{i,j},E^{k,l}]_{(m-1)}+m\delta_{m,1}\mathrm{str}(E^{i,j}E^{k,l})&
\end{align*}
we need only consider annihilation by operators $E_{(m)}^{i,j}, H^{i,j}_{(m)}\in\widehat{\mathfrak{n}}_+$ with $j=1$ or $i=2n-1,2n$. In Appendix~\ref{app:computationsforsingular}, we show that all of the relevant operators annihilate $\chi$ and so we conclude that $\chi$ is singular in $V^1(\psl_{n|n})$. Since $\chi\neq0$ and all positive modes annihilate $\chi$ which has degree 2, we can also conclude that $U\neq V^1(\psl_{n|n})$ and $U\neq 0$.
\end{proof}
As mentioned in Section \ref{pslsection} the vertex algebra $V^1(\mathfrak{psl}_{n|n})$ contains two even subalgebras $V^{1}(\mathfrak{sl}_n)$~and~$V^{-1}(\mathfrak{sl}_n)$ with commuting actions. The singular vectors generating maximal submodules for these two subalgebras are known for $n>1$ and $n>3$, respectively~\cite{arakawaSheetsAssociatedVarieties2019}. For $n>1$, the image of the singular vector for $V^1(\mathfrak{sl}_n)$ in $V^1(\mathfrak{psl}_{n|n})$ is
\begin{align*}
    &&\chi_+&=(E^{1,n}_{(-1)})^2\ket{0}&
\end{align*}
whilst for $n>3$ the image of the singular vector for the $V^{-1}(\mathfrak{sl}_n)$ subalgebra is
\begin{align*}
    &&\chi_-&=(E^{n+1,2n}_{(-1)}E^{n+2,2n-1}_{(-1)}-E^{n+1,2n-1}_{(-1)}E^{n+2,2n}_{(-1)})\ket{0}.&
\end{align*}
For these two vectors, we can easily prove the following lemma.
\begin{lemma}\label{lem:subsing}
    Although the vectors $\chi_+$ and $\chi_-$ are not singular vectors of $V^1(\psl_{n|n})$ we have  $\chi_{\pm} \in \langle \chi \rangle = U$.
\end{lemma}
\begin{proof}
It can be checked by an  explicit computation that $\chi_+$, $\chi_-$ are not annihilated by (for instance) $E^{n,n+1}_{(0)}$ and $E^{1,n+1}_{(0)}$, respectively. Therefore, they are not singular vectors of $V^1(\psl_{n|n})$. However, for $n>1$ we have
\begin{align*}
    &&\chi_+ &= E^{2n-1,n}_{(1)} T E^{2n,n}_{(1)} T \chi,&
\end{align*}
whereas for $n>3$
\begin{align*}
    &&\chi_-&=E^{n+2,1}_{(1)} T E^{n+1,1}_{(1)} T \chi .&
\end{align*}
\end{proof}
Before going on to study the associated variety of $\psl_{n|n}$, it appears natural to formulate the following conjecture:
\begin{conjecture}
 The maximal proper submodule of $V^1(\psl_{n|n})$ is $U$. Equivalently
 \begin{align*}
     &&L_1(\psl_{n|n})&=V^1(\psl_{n|n})/U.&
 \end{align*}
\end{conjecture}
One possible way to probe whether this conjecture is true and prove it is to employ the inverse quantum Hamiltonian reduction methods used, for example, in \cite{fehilyInverseReductionHooktype2023}. Alternatively, one may try to extend results of~\cite{Gotz:2006qp} to cases that cover the level of interest here.

\section{Associated variety of $L_1(\mathfrak{psl}_{n|n})$}
The strategy to proving the main theorem is to use our knowledge of the associated varieties of the $L_{\pm1}(\sli_n)$ subalgebras and incorporate the fact that the submodule $U$ must at least be contained in the maximal submodule of $V^1(\psl_{n|n})$. 

\subsection{Associated varieties of $L_{\pm 1} (\sli_n)$} We start by stating some known facts about the associated varieties of $L_{\pm 1}(\sli_n)$, and by deriving some immediate consequences for the associated variety of $L_1(\psl_{n|n})$.
\begin{theorem}[\cite{arakawaSheetsAssociatedVarieties2019}]\label{assvarsheet}
For $n>3$
\begin{align*}
     &&X_{L_{-1}(\mathfrak{sl}_n)}&=\overline{\Sh(\mathfrak{sl}_n)}.&
 \end{align*}
 For $n=2$ and $n=3$ it is an $SL_n$-invariant closed subvariety of $\sli_n^*$.
\end{theorem}
\begin{theorem}[\cite{arakawaRemarkC_2Cofiniteness2010}]
    \begin{align*}
    && X_{L_1(\mathfrak{sl}_n)}&=\lbrace pt\rbrace.&
\end{align*}
\end{theorem}
Since we have shown that for $n>1$ and $n>3$ the singular vectors $\chi_+$ and $\chi_-$ are in $U$, respectively, we can conclude the following.
\begin{proposition}\label{contained}
For $n\in \{2,3\}$, set
\begin{equation*}
    \mathbb{S} =
        \sli_n^*
\end{equation*}
whereas for $n>3$
\begin{equation*}
       \mathbb{S} = \overline{\Sh(\sli_n)} .
\end{equation*}
Then for all $n>1$, $X_{L_1(\mathfrak{psl}_{n|n})}$ is a closed $SL_n$-invariant subvariety of $\mathbb{S}$.
\end{proposition}
\begin{proof}
First, note that according to Example~\ref{ex:affineassoc}, there is a surjective map
\begin{align*}
    &&\pi_{L_1(\psl_{n|n})}:\C[\psl_{n|n}^*]&\twoheadrightarrow L_1(\psl_{n|n})/(t^{-2}\psl_{n|n}[t^{-1}]L_1(\psl_{n|n}))&
\end{align*}
such that
\begin{align*}
    &&X_{L_1(\psl_{n|n})}&\cong \mathrm{specm}\; I_{L_1(\psl_{n|n})}.&
\end{align*}
Following the remarks in Example~\ref{ex:affineassoc}, since odd elements are nilpotent, this is a closed $SL_{n} \times SL_{n}$-invariant subvariety of $(\psl^*_{n|n})_{\bar{0}} \cong \sli_{n}^*\times \sli_n^* \cong \sli_{n} \times \sli_n$, where we use the identification of Remark~\ref{rem:iden}. We also have surjective maps for each subalgebra
\begin{align*}
    &&\pi_{L_{\pm1(\sli_n)}}:\C[\sli_n]&\twoheadrightarrow L_{\pm1}(\sli_n)/(t^{-2}\sli_n[t^{-1}]L_{\pm1}(\sli_n))&
\end{align*}
such that
\begin{align*}
    &&X_{L_{\pm1}(\sli_n)}&\cong \mathrm{specm}\ker \pi_{L_{\pm1}(\sli_n)}.&
\end{align*}
We showed in Lemma~\ref{lem:subsing} that for $n>3$ the singular vectors $\chi_+$ and $\chi_-$ generating maxmial submodules of $V^{\pm 1}(\sli_n)$ are in $U$, and therefore in the maximal submodule of $V^{1}(\psl_{n|n})$. Thus, we can conclude from Lemma~\ref{lem:submodule}
\begin{align*}
    && I_{L_1(\sli_n)} \times I_{L_{-1}(\sli_n)}  &\subseteq I_{L_1(\psl_{n|n})},&
\end{align*}
so that 
\begin{align*}
    &&X_{L_1(\psl_{n|n})} &\subseteq \{pt\} \times \overline{\Sh (\sli_n)}\cong \mathbb{S}.&
\end{align*}
Similarly for $n\in \{1,2\}$ we can conclude that 
\begin{align*}
    && I_{L_1(\sli_n)} \times I_{V^1(\sli_n)}  &\subseteq I_{L_1(\psl_{n|n})},&
\end{align*}
which also implies 
\begin{align*}
    &&X_{L_1(\psl_{n|n})} &\subseteq \{pt\} \times \sli_n \cong \mathbb{S}.&
\end{align*}
\end{proof}

\subsection{Associated variety of $L_1(\psl_{n|n})$} Physical expectations from the 3d A-model for SQED[$n$] tell us that we should have
\begin{align*}
     &&X_{L_1(\mathfrak{psl}_{n|n})}&=\overline{\Omin(\mathfrak{sl}_n)}.&
\end{align*}
Notice that Proposition~\ref{contained} is compatible with this expectation, because as discussed in Section~\ref{subsubsec:sheetsA}, we have $\overline{\Omin(\mathfrak{sl}_n)}\subset\overline{\Sh(\mathfrak{sl}_n)} \subset \sli_n\cong \sli_n^*$.\\

In order to prove that the physical expectation holds, we first find vectors in $U$ (and so in the maximal submodule of $V^1(\psl_{n|n})$), whose image in the reduced $C_2$-algebra correspond to all $2\times 2$ minors of $\sli_n$ discussed in Remark~\ref{rem:iden}. These include those minors that for $n>3$ do not vanish on $\overline{\Sh(\sli_n)}$, a distinction that for completeness we discuss in Appendix~\ref{app:B}.
\begin{proposition}
  \label{nilprop}
    $X_{V^1(\mathfrak{psl}_{n|n})/U}\subseteq \overline{\Omin(\sli_n)}$.
\end{proposition}
\begin{proof}
Consider the map
\begin{align*}
    &&\Psi:\V&\to R_\V=\C[\psl_{n|n}^*]&\\
    &&v&\mapsto \begin{cases}
        X_{\alpha_1}\dots X_{\alpha_k}, & v=E^{\alpha_1}_{(-1)}\dots E^{\alpha_k}_{(-1)}\ket{0}\notin C_2(V^1(\psl_{n|n})) \\
        0 ,& v\in C_2(V^1(\psl_{n|n}))
    \end{cases}&
\end{align*}
where the $X_\alpha$ correspond to the coordinates of elementary matrices in $\mathrm{End}(\C^{n|n})$ with indices matching those of $\alpha$. We then denote by $\tilde{\Psi}=\sigma\circ \Psi$ the composition with the reduction map.\\\\
For $n+1\leq i< k \leq 2n$ , $n+1\leq j< l \leq 2n-1$ define the vector
\begin{align*}
    &&u_{i,k,j,l}=&E^{k,1}_{(1)}TE^{i,1}_{(1)}TE^{2n,l}_{(0)}E^{2n-1,j}_{(0)}\chi&,
\end{align*}
whereas for $n+1\leq i< k \leq 2n-1$ , $n+1\leq j< 2n-1$  $l=2n$ the small variation thereof
\begin{align*}
    &&u_{i,k,j,l}&=E^{k,1}_{(1)}TE^{i,1}_{(1)}TE^{2n-1,j}_{(0)}\chi,&
\end{align*}
and finally for $j=2n-1$ and $l=2n$
\begin{align*}
&&u_{i,k,j,l}&=E^{2n,1}_{(1)}TE^{2n-1,1}_{(1)}T\chi.&
\end{align*}
Compute for instance $u_{i,j,j,i}$ for $i<j<2n$:
\begin{align*}
&&u_{i,j,j,i}=&E^{j,1}_{(1)}TE^{i,1}_{(1)}TE^{2n,j}_{(0)}E^{2n-1,i}_{(0)}\chi&\\
    &&&=(D^{i,1}_{(-1)}D^{j,1}_{(-1)}-E^{j,i}_{(-1)}E^{i,j}_{(-1)}+E^{j,1}_{(-1)}E^{1,j}_{(-1)}-E^{1,i}_{(-1)}E^{i,1}_{(-1)}&\\
    &&&\ \ \ \ -D^{i,1}_{(-2)}-D^{j,1}_{(-2)}-E^{i,1}_{(-2)}-E^{1,j}_{(-2)} )\ket{0}.
\end{align*}
Since we ultimately only consider even terms with no modes of degree greater than $1$, only the first two terms will remain under application of $\Tilde{\Psi}$, and we get
\begin{align*}
    &&\tilde{\Psi}(u_{i,j,j,i})&=X_{i,j}X_{j,i}-X_{j,j}X_{i,i},
    &
\end{align*}
which are some of the desired minors of $\sli_n$ discussed in~\ref{subsec:nil}, Remark~\ref{rem:iden} and Appendix~\ref{app:B}.\footnote{Notice that due to the range of $i,j,k,l$, these minors are in $\sli_n \cong R_{V^1(\sli_n)}$, as required. In the decomposition outlined in Appendix~\ref{app:B}, these are minors in $V_{1,2}$ that do not vanish on $\overline{\Sh (\sli_n) } $.} A similar, albeit simpler computation shows that the following holds for all $n+1\leq i< k \leq 2n$, $n+1\leq j< l \leq 2n$
\begin{align*}
    &&\tilde{\Psi}(u_{i,k,j,l})&=X_{i,j}X_{k,l}-X_{k,j}X_{i,l}
    ,&
\end{align*}
covering all of the $2\times 2$ minors of $\sli_n$. By construction $u_{i,k,j,l}\in U$ and so $\tilde{\Psi}(u_{i,k,j,l})\in I_{V^1(\psl_{n|n})/U}$. Thus, all of the $2\times 2$ minors must vanish on $X_{\V/U}$. The result follows.
\end{proof}
We are now in a position to prove the main theorem:
\begin{theorem}
For $n>1$
\begin{align*}
&&X_{L_1(\mathfrak{psl}_{n|n})}&=\overline{\Omin(\sli_n)}.&
\end{align*}   
\end{theorem}
\begin{proof}
It is known that $L_1(\mathfrak{psl}_{n|n})$ is not $C_2$-cofinite \cite[\S 0.4]{gorelikSimplicityVacuumModules2006}. This implies $\dim X_{L_1(\mathfrak{psl}_{n|n})}\neq 0$ so we can conclude $X_{L_1(\mathfrak{psl}_{n|n})}\neq \lbrace pt \rbrace$. Moreover, by Proposition~\ref{nilprop}, $X_{L_1 (\psl_{n|n})}\subseteq X_{V^1(\psl_{n|n})/U} \subseteq \overline{\Omin}$. The only non-trivial $SL_n$-invariant subvarieties of $\overline{\Omin(\sli_n)}$ are itself and $\Omin(\sli_n)$. However $X_{L_1(\mathfrak{psl}_{n|n})}$ must be closed, hence the result.
\end{proof}
\begin{corollary} For all $n>1$ the VOA $L_1(\psl_{n|n})$ is quasi-lisse.
\end{corollary}
\begin{proof}
    This follows since $\overline{\Omin(\sli_n)}$ is known to have finitely many symplectic leaves.
\end{proof}

\subsection{Symplectic dual pairs from $L_1(\psl_{n|n})$} We conclude with a remark on symplectic duality. As mentioned in the introduction, an immediate application of Theorem~8.32 and Proposition~8.39 of~\cite{ballin3dMirrorSymmetry2023} gives the following
\begin{theorem}[\cite{ballin3dMirrorSymmetry2023}]
For SQED[$n$], with $\mathbf{1}_\rho^A\in \mathcal{C}_\rho^A$ the tensor unit in a suitable category of modules $\mathcal{C}_\rho^A$ for $V^A_\rho$, we have
 \begin{align*}
     \mathrm{Ext}^\bullet \left(\mathbf{1}_{\rho}^A,\mathbf{1}_{\rho}^A\right) \cong \mathbb{C}[A_{n-1}].&
 \end{align*}
\end{theorem}
\begin{proof}
    It is shown in~\cite{ballin3dMirrorSymmetry2023}, Theorem~8.32 that $\mathcal{C}_\rho^A \cong \mathcal{C}_{\rho^\vee}^B$, where $\mathcal{C}_{\rho^\vee}^B$ the category of modules for yet another VOA $V_{\rho^\vee}^B$ (see also Remark~\ref{rem:mirror}). Moreover, according to Proposition~8.39 the self-extensions of the tensor unit in this category is the coordinate ring on the \emph{Higgs branch} of the theory determined by $\rho^\vee$, see~\eqref{eq:higgs} and the subsequent discussion. Unpacking the definitions, this Higgs branch is the $A_{n-1}$ singularity.
\end{proof}
Thus, we can extract very concretely the symplectic dual pair $\overline{\Omin(\sli_n)}$ and $A_{n-1}$ singularity from the VOA $L_1(\psl_{n|n})$.

\appendix
\section{}\label{app:computationsforsingular}
We outline in this appendix the computations necessary to demonstrate that the vector $\chi$ is annihilated by the elements of $\widehat{\mathfrak{n}}_+\subset \V$. For degree zero elements, we only have operators $E^{i,j}_{(0)}$ with $i<j$ and so as suggested in the proof to Proposition \ref{singularvectorprop}, we only need to check the following:
\begin{align*}
    &&E_{(0)}^{2n-1,2n}\chi&=(-E^{1,2n}_{(-1)}E^{1,2n}_{(-1)}+E^{1,2n-1}_{(-1)}E^{2n-1,2n}_{(0)}E^{1,2n}_{(-1)})\ket{0}&\\
    &&&=0.&
\end{align*}
For elements of positive degree $m>0$, we need to consider the following where $j\neq 1$ and $i\neq 1, 2n,2n-1$ we can compute:
\begin{align*}
    &&E^{2n-1,j}_{(m)}\chi&=(-E^{1,j}_{(m-1)}E^{1,2n}_{(-1)}\pm E^{1,2n-1}_{(-1)}E^{2n-1,j}_{(m)}E^{1,2n}_{(-1)})\ket{0}&\\
    &&&=0&\\
    &&E^{2n,j}_{(m)}\chi&=\pm E^{1,2n-1}_{(-1)}E^{2n,j}_{(m)}E^{1,2n}_{(-1)}\ket{0}&\\
    &&&=\mp E^{1,2n-1}_{(-1)}E^{1,j}_{(m-1)}\ket{0}&\\
    &&&=0&\\
    &&E^{i,1}_{(m)}\chi&=(E^{i,2n-1}_{(m-1)}E^{1,2n}_{(-1)}\pm E^{1,2n-1}_{(-1)}E^{i,1}_{(m)}E^{1,2n}_{(-1)})\ket{0}&\\
    &&&=\pm E^{1,2n-1}_{(-1)}E^{i,2n}_{(m-1)}\ket{0}&\\
    &&&=0&\\
     &&E^{2n-1,1}_{(m)}\chi&=(D^{2n-1,1}_{(m-1)}E^{1,2n}_{(-1)}-\delta_{m,1}E^{1,2n}_{(-1)}-E^{1,2n-1}_{(-1)}E^{2n-1,1}_{(m)}E^{1,2n}_{(-1)})\ket{0}&\\
    &&&=(E^{1,2n}_{(m-2)}-\delta_{m,1}E^{1,2n}_{(-1)}-E^{1,2n-1}_{(-1)}E^{2n-1,2n}_{(m-1)})\ket{0}&\\
    &&&=0&
\end{align*}
where we note the choice of $\pm$ in the above depends on the values of $i$ and $j$ which determines whether the corresponding element of $\widehat{\mathfrak{n}}_+$ is even or odd. We also have
\begin{align*}
    &&E^{2n,1}_{(m)}\chi&=(E^{2n,2n-1}_{(m-1)}E^{1,2n}_{(-1)}-E^{1,2n-1}_{(-1)}E^{2n,1}_{(m)}E^{1,2n}_{(-1)})\ket{0}&\\
    &&&=(-E^{1,2n-1}_{(m-2)}-E^{1,2n-1}_{(-1)}D^{2n,1}_{(m-1)}+\delta_{m,1}E^{1,2n-1}_{(-1)})\ket{0}&\\
    &&&=0&\\
    &&H^{2n-2,2n-1}_{(m)}\chi&=(E^{1,2n-1}_{(m-1)}E^{1,2n}_{(-1)}+E^{1,2n-1}_{(-1)}H^{2n-2,2n-1}_{(m)}E^{1,2n}_{(-1)})\ket{0}&\\
    &&&=0 &\\
    &&H^{2n-1,2n}_{(m)}\chi&=(-E^{1,2n-1}_{(m-1)}E^{1,2n}_{(-1)}+E^{1,2n-1}_{(-1)}H^{2n-1,2n}_{(m)}E^{1,2n}_{(-1)})\ket{0}&\\
    &&&=E^{1,2n-1}_{(-1)}E^{1,2n}_{(m-1)}\ket{0}&\\
    &&&=0&\\
    &&H^{1,2}_{(m)}\chi&=(E^{1,2n-1}_{(m-1)}E^{1,2n}_{(-1)}+E^{1,2n-1}_{(-1)}H^{1,2}_{(m)}E^{1,2n}_{(-1)})]\ket{0}&\\
    &&&=E^{1,2n-1}_{(-1)}E^{1,2n}_{(m-1)}\ket{0}&\\
    &&&=0.&
\end{align*}

\section{}\label{app:B}
In this appendix, we discuss the sheets $\overline{\Sh(\sli_n)}$ and their relations to the minimal nilpotent orbits $\overline{\Omin (\sli_n)}$. Recall from Section~\ref{subsubsec:sheetsA} that since the minimal nilpotent orbit of type A corresponds to the partition $\alpha=(2,1,\dots,1)$, then the Levi subalgebra entering the definition of the sheet is determined by the conjugate partition $\overline{\alpha}=(n-1,1)$. Let
\begin{align*}
    &&h_k&=\mathrm{diag}\big(\underbrace{0,\dots,0}_{k-1},1,-1,\underbrace{0,\dots,0}_{n-k-1}\big)&
\end{align*}
be the standard basis for $\mathfrak{h}\subset\sli_n$. The centre of the Levi subalgebra corresponding to the partition $\overline{\alpha}$ can be expressed explicitly as
\begin{align*}
    &&\mathfrak{z}(\levi)&=\left\lbrace Y_1\sum_{i=1}^{n-1} (n-i) h_{i}\mid Y_1\in\C\right\rbrace&\\
    &&&=\left\lbrace \left(\begin{array}{cccc}
        Y_1(n-1) & 0 & \dots & 0  \\
        0 & -Y_1 & & \vdots\\
        \vdots &   & \ddots & \vdots \\
        0 & \dots & \dots & -Y_1
    \end{array} \right) \Bigg| Y_1\in\C\right\rbrace.&
\end{align*}
The nilradical for the parabolic subalgebra containing $\levi$ is given explicitly by 
\begin{align*}
        &&\para_u=\sum_{i=1}^{n-1} \g^{\sum_{j=1}^i \alpha_j}&=\left\lbrace\left(\begin{array}{cccc}
        0 & Y_2 & \dots & Y_n  \\
        \vdots & \ddots &  & 0 \\
        \vdots  & &  \ddots & \vdots \\
        0 & \dots & \dots & 0 
    \end{array} \right)\Bigg| Y_i\in\C  \right\rbrace&
    \end{align*}
   Applying Lemma \ref{sheetlemma}, we arrive at the following description for the sheet of type A containing $\overline{\Omin(\sli_n)}$.
   \begin{align*} 
    && \overline{\Sh(\sli_n)}&=\left\lbrace R\left(\begin{array}{cccc}
        Y_1(n-1) & Y_2 & \dots & Y_n  \\
        0 & -Y_1&  & 0 \\
        \vdots  & &  \ddots & \vdots \\
        0 & \dots & \dots & -Y_1
    \end{array} \right)R^{-1}\Bigg| Y_i\in\C,\; R\in\mathrm{SL}_n \right\rbrace
\end{align*}

As mentioned in Section~\ref{subsubsec:sheetsA}, $\overline{\Omin(\sli_n)}$ is the zero-locus of the functions in $\mathbb{C}[\sli_n]$ given by the $2\times2$ matrix minors. In Section~4 of \cite{weymanEquationsConjugacyClasses1989} it was shown that the space of $2\times 2$ minors can be expressed as
\begin{align}\label{minors}
    &&\wedge^2E^*\otimes\wedge^2E&\cong V_{1,2}\oplus U_{2,2}&
\end{align}
where $E=\C^n$, $U_{2,2}$ is the irreducible $SL_n$ representation generated by the minor $Z_{1,n}Z_{2,n-1}-Z_{1,n-1}Z_{2,n}$ and $V_{1,2}$ is the unique $SL_n$-invariant complement to $U_{2,2}$. The latter is shown in \cite{weymanEquationsConjugacyClasses1989},~Section~4 to be isomorphic to $E^*\otimes E$, with the image of a basis element $e_i^* \otimes e_j$ under this isomorphism
\begin{equation*}
e_i^* \otimes e_j \mapsto \sum_{k=1}^{n} Z_{i,j}Z_{k,k} - Z_{k,j}Z_{i,k}.
\end{equation*}
It is in fact generated by elements of the form $e_i^* \otimes e_i$ as an $SL_n$ representations, and it can be checked that these elements do not vanish on the sheet. We now show that the elements of $U_{2,2}$ vanish on both $\overline{\Omin(\sli_n)}$ and $\overline{\Sh(\sli_n)}$. 
\begin{proposition}\label{vanishing}
 For any $n>3$, $f\in U_{2,2}$ and $Z\in\overline{\Sh(\sli_n)}$ we have $f(Z)=0$.
\end{proposition}
\begin{proof}
Set $a=-Y_1$ and define two vectors $v,u \in E$ as $v^T = (nY_1,Y_2,\dots Y_n)R^{-1}$, $u=(R_{1,1},\dots R_{n,1})^T$. Then we have
\begin{align*}
    &&Z&=a\mathrm{Id}+uv^T.&
\end{align*}
Let $g=Z_{1,n}Z_{2,n-1}-Z_{1,n-1}Z_{2,n}$ and note that since $uv^T$ is at most a rank 1 matrix we have $g(Z)=0$. Note that since $\overline{\Sh(\sli_n)}$ is $SL_n$-invariant, the $2\times 2$ minors that vanish on $\overline{\Sh(\sli_n)}$ form a subrepresentation of $\wedge^2E^*\otimes\wedge^2 E$. It follows that the elements of the subrepresentation generated by $g$ vanish on $\overline{\Sh(\sli_n)}$. As stated earlier, $g$ generates the irreducible $SL_n$ representation $U_{2,2}$ and so for all $f\in U_{2,2}$ we have $f(x)=0$.
\end{proof}
\printbibliography
\end{document}